\documentclass{amsart}

\usepackage[utf8]{inputenc} 
\usepackage[T1]{fontenc} 
\usepackage[english]{babel} 
\usepackage{mathscinet}

\theoremstyle{plain} 
\newtheorem{theorem}{Theorem} 
\newtheorem{lemma}[theorem]{Lemma} 
\newtheorem{corollary}[theorem]{Corollary}

\theoremstyle{remark}
\newtheorem*{remark}{Remark}

\DeclareMathOperator{\mre}{Re}

\begin{document} 
\title[Projecting onto Helson matrices in Schatten classes]{Projecting onto Helson matrices \\ in Schatten classes} 
\date{\today} 

\author{Ole Fredrik Brevig} 
\address{Department of Mathematical Sciences, Norwegian University of Science and Technology (NTNU), NO-7491 Trondheim, Norway} 
\email{ole.brevig@math.ntnu.no}

\author{Nazar Miheisi} 
\address{Department of Mathematics, King’s College London, Strand, London WC2R 2LS, United Kingdom} 
\email{nazar.miheisi@kcl.ac.uk}

\begin{abstract}
	A Helson matrix is an infinite matrix $A = (a_{m,n})_{m,n\geq1}$ such that the entry $a_{m,n}$ depends only on the product $mn$. We demonstrate that the orthogonal projection from the Hilbert--Schmidt class $\mathcal{S}_2$ onto the subspace of Hilbert--Schmidt Helson matrices does not extend to a bounded operator on the Schatten class $\mathcal{S}_q$ for $1 \leq q \neq 2 < \infty$. In fact, we prove a more general result showing that a large class of natural projections onto Helson matrices are unbounded in the $\mathcal{S}_q$-norm for $1 \leq q \neq 2 < \infty$. Two additional results are also presented. 
\end{abstract}


\maketitle

\section{Introduction}
Let $\gamma = (\gamma_k)_{k\geq0}$ be a sequence of complex numbers. A \emph{Hankel matrix} is an infinite matrix of the form 
\begin{equation}\label{eq:Hankel} 
	H_\gamma = \left(\gamma_{i+j}\right)_{i,j\geq0}. 
\end{equation}
We consider the matrices \eqref{eq:Hankel} as linear operators on $\ell^2(\mathbb{N}_0)$, where $\mathbb{N}_0=\{0,1,2,\dots\}$. The multiplicative analogues of Hankel matrices --- that is, matrices whose entries depend on the product rather than the sum of the coordinates --- are known as \emph{Helson matrices}. To be precise, a Helson matrix is an infinite matrix of the form 
\begin{equation}\label{eq:Helson} 
	M_\varrho = \left(\varrho_{mn}\right)_{m,n\geq1} 
\end{equation}
for some sequence of complex numbers $\varrho = (\varrho_k)_{k\geq1}$. In this case, we consider the matrices \eqref{eq:Helson} as linear operators on $\ell^2(\mathbb{N})$, where $\mathbb{N}=\{1,2,3,\dots\}$. 

Helson matrices, whose study was initiated in the papers \cite{Helson06, Helson10}, play a similar role in the analysis of Dirichlet series as (additive) Hankel matrices play in the analysis of holomorphic functions in the unit disk. As such, questions regarding whether or not classical results for Hankel matrices can be extended to the multiplicative setting have attracted considerable recent attention (see e.g.~\cite{BP15, MP18, OCS12, PP218}). This note deals with one such question.

Recall that a compact operator $A\colon\ell^2\to\ell^2$ is in the Schatten class $\mathcal{S}_q$ if its sequence of singular values $s(A)=(s_k(A))_{k\geq0}$ is in $\ell^q$ and in this case
\[\|A\|_{\mathcal{S}_q}=\|s(A)\|_{\ell^q}.\]
Note that the Hilbert--Schmidt class $\mathcal{S}_2$ is a Hilbert space with inner product 
\begin{equation}\label{eq:HS} 
	\langle A, B \rangle = \operatorname{Tr}(AB^\ast) = \sum_{i=0}^\infty \sum_{j=0}^\infty a_{i,j} \overline{b_{i,j}}. 
\end{equation}
The \emph{averaging projection} $P$ onto the set of Hankel matrices is defined by 
\begin{equation}\label{proj.Hankel} 
	P\colon \left(a_{i,j}\right)_{i,j\geq0} \mapsto H_\gamma, \qquad \gamma_k = \frac{1}{k+1}\sum_{i+j=k}a_{i,j}. 
\end{equation}
It is not hard to see that the restriction of $P$ to $\mathcal{S}_2$ is the orthogonal projection onto the subspace of Hilbert--Schmidt Hankel matrices. A well-known result due to Peller \cite{Peller82} (see also \cite[Ch.~6.5]{Peller03}) is that the averaging projection $P$ is bounded on the Schatten class $\mathcal{S}_q$ for every $1<q<\infty$. 

The main purpose of this note is to show that the analogous statement for Helson matrices is false. We therefore define the \emph{averaging projection} $\mathcal{P}$ onto Helson matrices by 
\begin{equation}\label{proj.Helson} 
	\mathcal{P}\colon \left(a_{m,n}\right)_{m,n\geq1} \mapsto M_\varrho, \qquad \varrho_k = \frac{1}{d(k)}\sum_{mn=k} a_{m,n}, 
\end{equation}
where $d(k)$ denotes the number of divisors of the integer $k$. As before it is clear that the restriction of $\mathcal{P}$ to $\mathcal{S}_2$ is the orthogonal projection onto the subspace of Hilbert--Schmidt Helson matrices. Our first result is the following: 
\begin{theorem}\label{thm.1} 
	The projection $\mathcal{P}$ is unbounded on $\mathcal{S}_q$ for every $1\leq q\neq 2 < \infty$. 
\end{theorem}

Although the natural projection $P$ given by \eqref{proj.Hankel} is unbounded on $\mathcal{S}_1$, there do exist bounded projections onto the trace class Hankel operators. Let $\varphi\colon\mathbb{N}\times \mathbb{N} \to \mathbb{R}$ be a non-negative function such that for every integer $k\geq0$ it holds that 
\begin{equation}\label{avg.conditions} 
	\sum_{i+j=k} \varphi(i,j) = 1.
\end{equation}
Consider the \emph{weighted averaging projection} $P_\varphi$ defined by 
\begin{equation}\label{proj.Hankel3} 
	P_\varphi\colon \left(a_{i,j}\right)_{i,j\geq0} \mapsto H_\gamma, \qquad \gamma_k = \sum_{i+j=k}\varphi(i,j) a_{i,j}. 
\end{equation}
The condition \eqref{avg.conditions} ensures that $P_\varphi$ is indeed a projection. For $\alpha\geq1$, consider
\[\frac{1}{(1-z)^\alpha} = \sum_{j=0}^\infty c_\alpha(j) z^j, \qquad c_\alpha(j) = \binom{j+\alpha-1}{j}.\]
The weight $\varphi_{\alpha,\beta}(i,j) = c_\alpha(i)c_\beta(j) / c_{\alpha+\beta}(i+j)$ satisfies the condition \eqref{avg.conditions} and the projection $P_{\varphi_{\alpha,\beta}}$ is bounded on $\mathcal{S}_1$ if $\alpha,\beta>1$ (see \cite[Ch.~6.5]{Peller03} and \cite{BW86}). Note that the averaging projection \eqref{proj.Hankel} corresponds to the endpoint case $\alpha=\beta=1$.

It is natural to ask if we can similarly find a weighted averaging projection onto Helson matrices which is bounded in $\mathcal{S}_q$ for some $1 \leq q\neq 2 < \infty$. We will show that if the weight function is multiplicative (see Section \ref{multiplicative} for the definition), this question has a negative answer.
\begin{theorem}\label{thm.2} 
	Let $\Phi\colon \mathbb{N}\times \mathbb{N}\to \mathbb{R}$ be a non-negative multiplicative function such that for every integer $k\geq 1$ it holds that 
	\begin{equation}\label{eq:avgHelson} 
		\sum_{mn=k} \Phi(m,n) = 1. 
	\end{equation}
	Define the weighted projection $\mathcal{P}_\Phi$ by 
	\begin{equation}\label{proj.Helson3} 
		\mathcal{P}_\Phi\colon \left(a_{m,n}\right)_{m,n\geq1} \mapsto M_\varrho, \qquad \varrho_k = \sum_{mn=k} \Phi(m,n) a_{m,n}. 
	\end{equation}
	Then $\mathcal{P}_\Phi$ is unbounded on $\mathcal{S}_q$ for every $1\leq q\neq 2 < \infty$. 
\end{theorem}

The Riemann zeta function can be represented, when $\mre{s}>1$, by an absolutely convergent Dirichlet series or by an absolutely convergent Euler product, 
\begin{equation}\label{eq:Riemannzeta} 
	\zeta(s) = \sum_{n=1}^\infty n^{-s} = \prod_p (1-p^{-s})^{-1}. 
\end{equation}
The Euler product is taken over the increasing sequence of prime numbers. For $\alpha\geq1$, the general divisor function $d_\alpha(n)$ is defined by
\[\zeta^\alpha(s) = \sum_{n=1}^\infty d_\alpha(n) n^{-s}.\]
Note that $d_2$ is the usual divisor function $d$ appearing in the projection \eqref{proj.Helson}. One family of weights that satisfy the assumptions of Theorem~\ref{thm.2} are
\[\Phi_{\alpha,\beta}(m,n) = \frac{d_\alpha(m)d_\beta(n)}{d_{\alpha+\beta}(mn)}\]
for $\alpha,\beta\geq1$. Note that the averaging projection \eqref{proj.Helson} again is equal to the endpoint case $\alpha=\beta=1$, and hence Theorem~\ref{thm.1} is a special case of Theorem~\ref{thm.2}.

\subsection*{Organization} The present note is organized into four sections. In Section~\ref{sec:prelim} we collect some preliminary material on infinite tensor products and multiplicative matrices. Section~\ref{sec:proof} is devoted to the proof of Theorem~\ref{thm.2}. The final section contains two additional results. The first is that there are no bounded projections from the spaces of compact and bounded operators to Helson matrices, while the second is a corollary of Theorem~\ref{thm.1} showing that the usual duality relation between Hankel matrices in $\mathcal{S}_q$ does not extend to Helson matrices.

\section{Infinite tensor products and multiplicative matrices} \label{sec:prelim}
In the present section we seek to represent $\ell^2(\mathbb{N})$ as an infinite tensor product of $\ell^2(\mathbb{N}_0)$. We will then discuss multiplicative matrices, with particular emphasis on Helson matrices. Our presentation and notation is inspired by \cite{Hilberdink17}. 

\subsection{Tensor product representation of $\ell^2(\mathbb{N})$} \label{sec:tensor}

For each prime $p$, consider the index set $\langle p \rangle = \{p^\kappa \,:\, \kappa \in \mathbb{N}_0\}$. It evidently holds that $\ell^2(\mathbb{N}_0) \simeq \ell^2(\langle p \rangle)$ through the obvious mapping. Note also that $\ell^2(\langle p \rangle)$ is a natural subspace of $\ell^2(\mathbb{N})$ since $\langle p \rangle \subseteq \mathbb{N}$. Let $(e_k)_{k\geq 1}$ (resp. $(e_k)_{k\geq0}$) denote the standard orthonormal basis of $\ell^2(\mathbb{N})$ (resp. $\ell^2(\mathbb{N}_0)$). Then $(e_{p^\kappa})_{\kappa\geq0}$ is an orthonormal basis of $\ell^2(\langle p \rangle)$; throughout we will identify each operator on $\ell^2(\langle p \rangle)$ with its matrix in this basis.

Let $\bigotimes_{p \leq p_N} \ell^2(\langle p \rangle)$ denote the Hilbert space tensor product of $\ell^2(\langle p \rangle)$ over the first $N$ primes. The linear extension of the map
\[\otimes_{p\leq p_N} x_p \mapsto \left( \otimes_{p\leq p_N} x_p \right)\otimes e_1\]
gives an embedding of $\bigotimes_{p \leq p_N} \ell^2(\langle p \rangle)$ into $\bigotimes_{p \leq p_{N+1}} \ell^2(\langle p \rangle)$. The inductive limit of this system as $N\to\infty$ can be identified with the linear span of all elements of the form $\otimes_p x_p$ such that only finitely many of the $x_p \in \ell^2(\langle p \rangle)$ are different from $e_1$. We can endow the limit with an inner product by setting 
\begin{equation}\label{inner-prod} 
	\langle \otimes_p x_p, \otimes_p y_p \rangle = \prod_p \langle x_p, y_p \rangle 
\end{equation}
and extending linearly. The infinite tensor product $\bigotimes_p \ell^2(\langle p \rangle)$ is defined to be the completion of the inductive limit with respect to the norm induced by the inner product \eqref{inner-prod}. 

Consider the prime factorization 
\begin{equation}\label{eq:primefact} 
	n = \prod_{p} p^{\kappa_p} 
\end{equation}
and note that for every integer $n\geq1$, it holds that $\kappa_p=0$ for all but a finite number of primes $p$. In view of \eqref{eq:primefact}, we define a linear map from $\ell^2(\mathbb{N})$ to $\bigotimes_p \ell^2(\langle p \rangle)$ by setting
\[e_n \mapsto \otimes_p e_{p^{\kappa_p}}.\]
It is easily seen that this map extends to a unitary operator and thus allows us to make the identification 
\begin{equation}\label{tensor} 
	\ell^2(\mathbb{N}) \simeq \bigotimes_p \ell^2(\langle p \rangle). 
\end{equation}

For each prime number $p$, let $\mathcal{R}_p$ denote the orthogonal projection from $\ell^2(\mathbb{N})$ to $\ell^2(\langle p \rangle)$, i.e. the operator defined by 
\begin{equation}\label{eq:Rp} 
	\mathcal{R}_p e_n = 
	\begin{cases}
		e_n & \text{if } n=p^\kappa, \\
		0 & \text{otherwise}, 
	\end{cases}
\end{equation}
and extending linearly. For a matrix $A \colon \ell^2(\mathbb{N})\to \ell^2(\mathbb{N})$, set $A_p = \mathcal{R}_p A \mathcal{R}_p^\ast$. We consider $A_p$ an operator on $\ell^2(\langle p \rangle)$ and note that its matrix can be obtained by deleting all rows and columns of $A$ whose index is not a power of $p$. It evidently holds that $\|A_p\| \leq \|A\|$ and the same estimate holds also for the $\mathcal{S}_q$-norms. Note that if $A$ is the Helson matrix \eqref{eq:Helson} generated by the sequence $\varrho=(\varrho_k)_{k\geq1}$, then $A_p$ is the Hankel matrix \eqref{eq:Hankel} generated by $\gamma=(\gamma_\kappa)_{\kappa\geq0} = (\varrho_{p^\kappa})_{\kappa\geq0}$.

\subsection{Multiplicative functions} \label{multiplicative} A function $F\colon\mathbb{N}\to \mathbb{C}$ is said to be \emph{multiplicative} if $F(1)=1$ and
\[F(mn)=F(m)F(n)\]
whenever $m$ and $n$ are coprime. Similarly, a function of two variables $f\colon\mathbb{N}\times\mathbb{N} \to \mathbb{C}$ is called \emph{multiplicative} if $f(1,1)=1$ and
\[f(m_1n_1,m_2n_2)=f(m_1,m_2)f(n_1,n_2)\]
whenever $m_1m_2$ and $n_1n_2$ are coprime. If $F \colon \mathbb{N}\to\mathbb{C}$ is multiplicative, then $f(m,n)=F(mn)$ is evidently also multiplicative. We shall also have use of the following basic result, which is certainly not new. However, we include a short proof for the benefit of the reader.
\begin{lemma}\label{lem:conv} 
	If $f \colon \mathbb{N}\times \mathbb{N}\to\mathbb{C}$ is multiplicative, then the convolution
	\[F(k) = \sum_{mn=k} f(m,n)\]
	is also multiplicative. 
\end{lemma}
\begin{proof}
	Suppose that $k$ and $l$ are coprime. If $mn=kl$, then we can factor $m= m_1 n_1$ and $n = m_2 n_2$ such that $m_1 m_2 = k$ and $n_1 n_2 = l$. Clearly $m_1 m_2$ and $n_1 n_2$ are coprime, and so it holds that 
	\begin{align*}
		F(kl) = \sum_{mn = kl} f(m,n) &= \sum_{\substack{m_1 m_2 = k \\
		n_1 n_2 = l}} f(m_1 n_1, m_2 n_2) \\
		&= \sum_{\substack{m_1 m_2 = k \\
		n_1 n_2 = l}}f(m_1,m_2)f(n_1,n_2) = F(k)F(l) 
	\end{align*}
	as desired. 
\end{proof}

\subsection{Multiplicative matrices} \label{sec:multmat} For every prime $p$ let $A_p$ be a bounded linear operator on $\ell^2(\langle p \rangle)$. If $\prod_p \|A_p\|$ converges, and each of the sums 
\[\sum_p \big|\|A_p e_1\|-1\big| \qquad \text{and} \qquad \sum_p\big|\langle A_p e_1, e_1 \rangle -1\big|\]
also converge, then the infinite tensor product $\bigotimes_p A_p$ defines a bounded operator on $\bigotimes_p \ell^2(\langle p \rangle)$ (see e.g.~\cite[Prop.~6]{Guichardet69}). Suppose in addition that $A_p\in\mathcal{S}_q$ for each $p$, and $\bigotimes_p A_p\in\mathcal{S}_q$. Then as a consequence of \cite[Thm.~2.4]{Hilberdink17} we have that 
\begin{equation} \label{cross-norm}
	\Big\|\bigotimes_p A_p\Big\|_{\mathcal{S}_q} = \prod_p \|A_p\|_{\mathcal{S}_q}. 
\end{equation}
We remark that the identity \eqref{cross-norm} is also valid for the operator norm. By the identification \eqref{tensor} we can regard $A = \bigotimes_p A_p$ as an operator on $\ell^2(\mathbb{N})$.

A matrix $A = (a_{m,n})_{m,n\geq1}$ is called \emph{multiplicative} if there is a multiplicative function $f \colon \mathbb{N}\times \mathbb{N}\to\mathbb{C}$ such that $a_{m,n}=f(m,n)$. In the case $A = \bigotimes_p A_p$ discussed above, it is easily verified that $A$ is multiplicative if $\langle A_p e_1, e_1 \rangle=1$ for every $p$. Note that in this case, we also have $A_p = \mathcal{R}_p A \mathcal{R}_p^\ast$ where $\mathcal{R}_p$ is as in \eqref{eq:Rp}. Conversely, if $A$ is multiplicative, then we have $A = \bigotimes_p A_p$, where again $A_p = \mathcal{R}_p A \mathcal{R}_p^\ast$.

Returning to the case of Helson matrices, we find that a Helson matrix $M_\varrho$ is multiplicative if and only if $\varrho_k = F(k)$ for a multiplicative function $F$. As mentioned in Section~\ref{sec:tensor}, in this case $\mathcal{R}_p M_\varrho \mathcal{R}_p^\ast = H_{\gamma}$ where $\gamma_j = F(p^j)$.

\section{Proof of Theorem~\ref{thm.2}} \label{sec:proof}
The proof of Theorem~\ref{thm.2} is inspired by the counter-examples to Nehari's theorem for Helson matrices constructed in \cite{BP15,OCS12}. We will demonstrate that any weighted averaging projection \eqref{proj.Hankel3} onto Hankel matrices cannot be contractive on $\mathcal{S}_q$ for $1 \leq q \neq 2 < \infty$. Specifically, we will prove that there is a universal lower bound for the norm of $P_\varphi$ on $\mathcal{S}_q$ which is strictly greater than $1$. 

If $\Phi$ is multiplicative, then the projection $\mathcal{P}_\Phi A$ given by \eqref{proj.Helson3} will preserve the tensor structure $A = \bigotimes_p A_p$ of a multiplicative matrix and factor into a tensor product of the projections $P_{\varphi_p} A_p$ given by \eqref{proj.Hankel3}, for some weight functions $\varphi_p$. The result will then follow from a standard argument.

Note that for the projection $P$ given by \eqref{proj.Hankel}, it is not hard to show, using Peller's criterion for Hankel operators of class $\mathcal{S}_q$ (see~\cite[Ch.~6.2]{Peller03}), that there is a constant $C$ such that 
\begin{align*}
	\|P\|_{\mathcal{S}_q\to\mathcal{S}_q} &\geq \frac{C}{\sqrt{q-1}} \intertext{as $q\to 1^+$. By a duality argument it also follows that as $q\to\infty$ we have} \|P\|_{\mathcal{S}_q\to\mathcal{S}_q} &\geq C \sqrt{q}. 
\end{align*}
In particular, the projection $P$ cannot be a contraction on $\mathcal{S}_q$ for $q$ sufficiently close to 1 or $q$ sufficiently large. The key point of the following result therefore is that this also holds for $q$ close to $2$ and that the lower bound holds uniformly for all weighted averaging projections.
\begin{lemma}\label{lem:1p2} 
	Fix $1\leq q\neq2 < \infty$. There exists some $\delta=\delta_q>0$ such that for every non-negative function $\varphi\colon\mathbb{N} \times \mathbb{N} \to \mathbb{R}$ satisfying \eqref{avg.conditions}, the weighted averaging projection $P_\varphi$ given by \eqref{proj.Hankel3} satisfies the bound $\|P_\varphi\|_{\mathcal{S}_q\to\mathcal{S}_q} \geq 1+\delta.$ 
\end{lemma}
The proof consists of three parts. We first compile some preliminary information. The two cases $1 \leq q < 2$ and $2<q<\infty$ will then be handled separately, but by fairly similar arguments. 
\begin{proof}
	For non-negative real numbers $t$ we will consider the following matrices:
	\[A(t) = 
	\begin{pmatrix}
		1 & 0 & 0 \\
		0 & t & 0 \\
		0 & 0 & 0 
	\end{pmatrix}
	\qquad B(t) = 
	\begin{pmatrix}
		1 & 0 & t \\
		0 & 0 & 0 \\
		0 & 0 & 0 
	\end{pmatrix}
	\]
	\[C(t) = 
	\begin{pmatrix}
		1 & 0 & t \\
		0 & 0 & 0 \\
		t & 0 & 0 
	\end{pmatrix}
	\qquad D(t) = 
	\begin{pmatrix}
		1 & 0 & t \\
		0 & t & 0 \\
		t & 0 & 0 
	\end{pmatrix}
	\]
	The singular values of $A(t)$ are $1$ and $t$, while $B(t)$ has only one singular value $\sqrt{1+t^2}$. A direct computation yields that the singular values of $C(t)$ are
	\[s(C(t)) = \left\{\frac{1}{2}+\sqrt{\frac{1}{4}+t^2},\,-\frac{1}{2}+\sqrt{\frac{1}{4}+t^2}\right\}.\]
	The same computation also yields that $s(D(t)) = s(C(t))\cup\{t\}$. We will only have need to refer to $\varphi(0,2)$, $\varphi(1,1)$ and $\varphi(2,0)$ and so for ease of notation we set
	\[\varphi_0 = \varphi(0,2), \qquad \varphi_1 = \varphi(1,1), \qquad \varphi_2 = \varphi(2,0).\]
	Recalling that $\varphi(0,0)=1$ we find that
	\[P_\varphi A(t) = D(\varphi_1 t), \qquad P_\varphi B(t) = D(\varphi_0 t), \qquad P_\varphi C(t) = D((\varphi_0 + \varphi_2) t ).\]
	
	Suppose that $1 \leq q < 2$. We consider $A(t)$ and find that 
	\begin{equation}\label{eq:1p2est1} 
		\|P_\varphi\|_{\mathcal{S}_q\to\mathcal{S}_q} \geq \lim_{t\to\infty} \frac{\|P_\varphi A(t)\|_{\mathcal{S}_q}}{\|A(t)\|_{\mathcal{S}_q}} = \lim_{t\to\infty} \frac{\|D(\varphi_1 t)\|_{\mathcal{S}_q}}{\|A(t)\|_{\mathcal{S}_q}}= 3^{1/q} \varphi_1. 
	\end{equation}
	We now consider $B(t)$. We estimate the $\mathcal{S}_q$-norm of $P_\varphi B(t) = D(\varphi_0 t)$ from below by considering only the two largest singular values, and noting that the largest is bounded below by $1$. Hence we obtain 
	\begin{equation}\label{eq:1p2est2} 
		\|P_\varphi\|_{\mathcal{S}_q\to\mathcal{S}_q} \geq \sup_{t\geq0} \frac{\|P_\varphi B(t)\|_{\mathcal{S}_q}}{\|B(t)\|_{\mathcal{S}_q}} \geq \sup_{t\geq0} \frac{\left(1+(\varphi_0 t)^q\right)^\frac{1}{q}}{\left(1+t^2\right)^\frac{1}{2}} \geq \big(1 + \varphi_0^\frac{2q}{2-q}\big)^\frac{2-q}{2q}, 
	\end{equation}
	where in the final estimate we chose $t = \varphi_0^{q/(2-q)}$. Considering the matrix transpose of $B(t)$ we see that the estimate \eqref{eq:1p2est2} also holds if $\varphi_0$ is replaced by $\varphi_2$. Recalling that $\varphi_0 + \varphi_1 + \varphi_2 = 1$, we conclude that $\varphi_1 \geq 1 -2x$ with $x = \max(\varphi_0,\varphi_2)$. Combining \eqref{eq:1p2est1} and \eqref{eq:1p2est2} we hence obtain the uniform lower bound
	\[\|P_\varphi\|_{\mathcal{S}_q\to\mathcal{S}_q} \geq \inf_{0 \leq x \leq 1} \max\left(3^{1/q}(1-2x), \big(1 + x^\frac{2q}{2-q}\big)^\frac{2-q}{2q} \right) = \big(1 + x_q^\frac{2q}{2-q}\big)^\frac{2-q}{2q},\]
	where $0<x_q<1$ denotes the unique positive solution of the equation 
	\begin{equation}\label{eq:1p2lb} 
		3^{1/q}(1-2x) = \big(1 + x^\frac{2q}{2-q}\big)^\frac{2-q}{2q}. 
	\end{equation}
	This completes the proof in the case $1 \leq q < 2$. 
	
	Next, we suppose that $2<q<\infty$. We consider $C(t)$ and after recalling that $P_\varphi C(t) = D((\varphi_0+\varphi_2)t)$, we obtain the lower bound 
	\begin{equation}\label{eq:1p2est3} 
		\|P_\varphi\|_{\mathcal{S}_q\to\mathcal{S}_q} \geq \lim_{t\to\infty} \frac{\|P_\varphi C(t)\|_{\mathcal{S}_q}}{\|C(t)\|_{\mathcal{S}_q}} = \left(\frac{3}{2}\right)^{1/q} (\varphi_0+\varphi_2). 
	\end{equation}
	We now consider $A(t)$ and estimate $P_\varphi A(t) = D(\varphi_1 t)$ from below by considering only the largest singular value and using a trivial inequality, to obtain
	\[\|P_\varphi A(t)\|_{\mathcal{S}_q} \geq \frac{1}{2}+\sqrt{\frac{1}{4}+(\varphi_1 t)^2} \geq \sqrt{1+(\varphi_1 t)^2}.\]
	Hence we find that 
	\begin{equation}\label{eq:1p2est4} 
		\|P_\varphi\|_{\mathcal{S}_q\to\mathcal{S}_q} \geq \sup_{t\geq0} \frac{\|P_\varphi A(t)\|_{\mathcal{S}_q}}{\|A(t)\|_{\mathcal{S}_q}} \geq \sup_{t\geq0} \frac{\left(1+(\varphi_1 t)^2\right)^\frac{1}{2}}{\left(1+t^q\right)^\frac{1}{q}} \geq \big(1 + \varphi_1^\frac{2q}{q-2}\big)^\frac{q-2}{2q}, 
	\end{equation}
	where we in the final estimate chose $t = \varphi_1^{2/(q-2)}$. Recalling that $\varphi_0+\varphi_2 = 1 - \varphi_1$ and setting $x=\varphi_1$, we combine \eqref{eq:1p2est3} and \eqref{eq:1p2est4} to obtain
	\[\|P_\varphi\|_{\mathcal{S}_q\to\mathcal{S}_q} \geq \inf_{0 \leq x \leq 1} \max\left(\left(\frac{3}{2}\right)^{1/q}(1-x), \big(1 + x^\frac{2q}{q-2}\big)^\frac{q-2}{2q} \right) = \big(1 + x_q^\frac{2q}{q-2}\big)^\frac{q-2}{2q},\]
	where $0<x_q<1$ denotes the unique positive solution of the equation
	\[\left(\frac{3}{2}\right)^{1/q}(1-x) = \big(1 + x^\frac{2q}{q-2}\big)^\frac{q-2}{2q}.\]
	This completes the proof in the case $2<q<\infty$. 
\end{proof}
\begin{remark}
	We can solve the equation \eqref{eq:1p2lb} for $q=1$ and obtain the explicit lower bound
	\[\|P_\varphi\|_{\mathcal{S}_1\to\mathcal{S}_1}\geq \frac{3}{35}\left(4\sqrt{11}-1\right) = 1.0514142138\ldots\]
	which holds for all weighted averaging projections \eqref{proj.Hankel3}.
\end{remark}
\begin{proof}[Proof of Theorem~\ref{thm.2}] 
	For each prime $p$ set $\varphi_p(i,j)=\Phi(p^i, p^j)$. Since $\Phi$ satisfies \eqref{eq:avgHelson}, we see that $\varphi_p$ satisfies \eqref{avg.conditions}. Suppose that $A = \bigotimes_p A_p$ is a multiplicative matrix. Since the weight $\Phi$ is also assumed to be multiplicative, we find by Lemma~\ref{lem:conv} that the sequence
	\[\varrho_k = \sum_{mn=k} \Phi(m,n)a_{m,n}\]
	is multiplicative. This means that $\mathcal{P}_\Phi A$ is a multiplicative Helson matrix, and since clearly $\mathcal{R}_p \mathcal{P}_\Phi A \mathcal{R}_p^\ast = P_{\varphi_p} A_p$ by the discussion in Section~\ref{sec:prelim}, we get that
	\[\mathcal{P}_\Phi A = \bigotimes_p P_{\varphi_p} A_p.\]
	Fix a positive integer $N$. For $p \leq p_N$, we choose $A_p$ such that $\|A_p\|_{\mathcal{S}_q}=1$ and $\|P_{\varphi_p}A_p\|_{\mathcal{S}_q}\geq 1+\delta$, where $\delta>0$ depends only on $1 \leq q \neq 2 < \infty$. Observe that as a consequence of Lemma~\ref{lem:1p2}, we can always make such a choice for $A_p$. For $p>p_N$ we choose $A_p=H_{e_0}$ so that $P_{\varphi_p} A_p = H_{e_0}$. We then obtain from \eqref{cross-norm} that
	\[\|\mathcal{P}_\Phi\|_{\mathcal{S}_q\to \mathcal{S}_q} \geq (1+\delta)^N.\]
	Then letting $N\to\infty$ we see that $\mathcal{P}_\Phi$ is unbounded on $\mathcal{S}_q$. 
\end{proof}
\begin{remark}
	The weights $\Phi_{\alpha,\beta}$ and $\varphi_{\alpha,\beta}$ discussed in the introduction are related as in the proof of Theorem~\ref{thm.2}. Indeed, inspecting the Euler product of the Riemann zeta function \eqref{eq:Riemannzeta} we find that $d_\alpha(p^j)=c_\alpha(j)$ for every prime $p$ and every $j\geq0$. 
\end{remark}

\section{Additional results}

\subsection{Projections on spaces of compact and bounded operators} Consulting Theorem~5.11 and Theorem~5.12 in \cite[Ch.~6.5]{Peller03}, we recall that there are no bounded projections $P_\varphi$ from the space of compact (resp. bounded) operators onto the space of compact (resp. bounded) Hankel matrices. It is trivial to extend this result to Helson matrices, and in this case we do not require the weight to be multiplicative. 
\begin{theorem}\label{thm:boundedcompact} 
	There are no bounded projections from the space of compact (resp. bounded) operators onto the space of compact (resp. bounded) Helson matrices. 
\end{theorem}
\begin{proof}
	Clearly, a bounded projection $\mathcal{P}_\Phi$ must satisfy \eqref{eq:avgHelson}. Then $\varphi(i,j)=\Phi(2^i,2^j)$ satisfies \eqref{avg.conditions}. For any compact (resp. bounded) operator $A \colon \ell^2(\mathbb{N}_0)\to\ell^2(\mathbb{N}_0)$ define the operator $\widetilde{A} \colon \ell^2(\mathbb{N})\to\ell^2(\mathbb{N})$ by
	\[\widetilde{a}_{m,n} = 
	\begin{cases}
		a_{i,j} & \text{if } m=2^{i} \text{ and } n=2^{j}, \\
		0 & \text{otherwise}. 
	\end{cases}
	\]
	Since $\mathcal{P}_\Phi \widetilde{A} = \widetilde{P_\varphi A}$, we see that if $\mathcal{P}_\Phi$ acts boundedly on the space of compact (resp. bounded) operators on $\ell^2(\mathbb{N})$, then $P_\varphi$ acts boundedly on the space of compact (resp. bounded) operators on $\ell^2(\mathbb{N}_0)$. However, this is impossible by the results mentioned above. 
\end{proof}
\begin{remark}
	We actually have $\widetilde{A} = A \otimes H_{e_0} \otimes H_{e_0} \otimes \cdots$ as in the proof of Theorem~\ref{thm.2}. 
\end{remark}

\subsection{Duality} We fix $1<q<\infty$ and set $1/q+1/r=1$. It is a standard fact that $(\mathcal{S}_q)^\ast\simeq\mathcal{S}_r$ with respect to the pairing arising from the inner product \eqref{eq:HS} of $\mathcal{S}_2$, i.e. the pairing $\langle A, B \rangle = \operatorname{Tr}(AB^\ast)$ for $A \in \mathcal{S}_q$ and $B \in \mathcal{S}_r$.

Let $H\mathcal{S}_q$ and $M\mathcal{S}_q$ denote the spaces of Hankel matrices and Helson matrices respectively of class $\mathcal{S}_q$. It is well-known that the pairing \eqref{eq:HS} also yields the duality $(H\mathcal{S}_q)^\ast\simeq H\mathcal{S}_r$. Clearly, the map $M_\varrho \mapsto \langle \cdot, M_\varrho \rangle$, is an embedding of $M\mathcal{S}_r$ into $(M\mathcal{S}_q)^\ast$. We now show that in contrast to Hankel matrices, this is not an isomorphism unless $q=2$. 
\begin{corollary}\label{duality} 
	Let $1<q\neq2<\infty$ and set $1/q+1/r=1$. The map $M_\varrho \mapsto \langle \cdot, M_\varrho \rangle$ from $M\mathcal{S}_r$ to $(M\mathcal{S}_q)^\ast$ is not surjective. 
\end{corollary}

Before proceeding, we fix some notation. For a subset $X\subseteq \mathcal{S}_q$, we denote by $X^\perp$ the \emph{annihilator} of $X$ in $\mathcal{S}_r$, i.e. $X^\perp$ consists of all $B\in\mathcal{S}_r$ such that $\langle A, B \rangle =0$ for all $A\in X$. 
\begin{proof}
	First observe that for a Helson matrix $M_\varrho \in\mathcal{S}_q$ and $A=\left(a_{m,n}\right)_{m,n\geq 1}\in\mathcal{S}_r$ we have that
	\[\langle M_\varrho, A \rangle = \sum_{m=1}^\infty \sum_{n=1}^\infty \varrho_{mn} \overline{a_{m,n}} = \sum_{k=1}^\infty d(k) \varrho_k \frac{1}{d(k)} \sum_{mn=k} \overline{a_{m,n}} = \langle M_\varrho, \mathcal{P}A \rangle.\]
	Therefore $(M\mathcal{S}_q)^\perp = \operatorname{Ker} \mathcal{P}\cap\mathcal{S}_r$, where $\mathcal{P}$ is the averaging projection \eqref{proj.Helson}. In particular, this shows that $\operatorname{Ker} \mathcal{P}\cap\mathcal{S}_r$ is a closed subspace of $\mathcal{S}_r$. Suppose that $M_\varrho \mapsto \langle \cdot, M_\varrho \rangle$ is surjective. Then by the open mapping theorem we have the isomorphism $M\mathcal{S}_r\simeq (M\mathcal{S}_q)^\ast$, with the pairing \eqref{eq:HS}. By elementary functional analysis it follows that $M\mathcal{S}_r \simeq \mathcal{S}_r/(\operatorname{Ker} \mathcal{P}\cap\mathcal{S}_r)$ and so
	\[\mathcal{S}_r = M\mathcal{S}_r \oplus (\operatorname{Ker}\mathcal{P}\cap\mathcal{S}_r).\]
	However, this would imply that $\mathcal{P}$ is bounded on $\mathcal{S}_r$ (by e.g.~\cite[Thm.~5.16]{Rudin73}), contradicting Theorem \ref{thm.1}. 
\end{proof}

\bibliographystyle{amsplain} 
\bibliography{helsonproj} 
\end{document}